\documentclass[12pt,a4paper,reqno,openright]{amsart}

\usepackage{amssymb,amsfonts,amsthm,amsmath,amsopn,amstext,amscd,latexsym}
\usepackage{bbm}
\usepackage{enumerate}
\usepackage{setspace}
\usepackage{hyperref}
\hypersetup{colorlinks=true}
\usepackage[all]{xy}
\usepackage{geometry}
\usepackage{float, wrapfig}
\usepackage{color}
\usepackage{cancel}

\usepackage[pdftex]{graphicx}

\newcommand{\bbG}{{\mathbb G}}

\newcommand{\cF}{{\mathcal F}}

\newcommand{\N}{\mathbb N}
\newcommand{\Z}{\mathbb Z}
\newcommand{\Q}{\mathbb Q}
\newcommand{\R}{\mathbb R}

\newcommand{\ra}{\rightarrow}
\newcommand{\lra}{\longrightarrow}

\newcommand{\Ra}{\Rightarrow}

\renewcommand{\1}{\mathbbm{1}}
\newcommand{\inv}{^{-1}}
\newcommand{\g}{\gamma}
\newcommand{\G}{\Gamma}

\DeclareMathOperator{\rk}{rk}
\DeclareMathOperator{\id}{id}

\DeclareMathOperator{\Aut}{Aut}
\DeclareMathOperator{\Ext}{Ext}
\DeclareMathOperator{\Int}{Int}

\DeclareMathOperator{\im}{im}

\DeclareMathOperator{\lex}{lex}

\DeclareMathOperator{\Hom}{Hom}
\DeclareMathOperator{\End}{End}
\newcommand{\vAut}{v\text{-}\Aut}
\newcommand{\oAut}{o\text{-}\Aut}

\newcommand{\adAut}{\Aut^S}

\newcommand{\GG}{\bbG}

\DeclareMathOperator*\HH{\mathbf{H}}
\DeclareMathOperator*\supp{supp}

\renewcommand{\epsilon}{\varepsilon}
\renewcommand{\phi}{\varphi}
\renewcommand{\theta}{\vartheta}

\theoremstyle{definition}
\newtheorem{defn}{Definition}[section]
\theoremstyle{plain}
\newtheorem{lemma}[defn]{Lemma}
\theoremstyle{plain}
\newtheorem*{lemma*}{Lemma}
\theoremstyle{plain}
\newtheorem{prop}[defn]{Proposition}
\theoremstyle{plain}
\theoremstyle{plain}
\newtheorem*{prop*}{Proposition}
\newtheorem{theorem}[defn]{Theorem}
\theoremstyle{plain}
\newtheorem*{teorema*}{Theorem}
\theoremstyle{plain}
\newtheorem{cor}[defn]{Corollary}
\theoremstyle{plain}
\newtheorem*{cor*}{Corollary}
\theoremstyle{definition}
\newtheorem{es}[defn]{Example}
\theoremstyle{definition}

\theoremstyle{plain}

\theoremstyle{definition}

\theoremstyle{definition}
\newtheorem{rmk}[defn]{Remark}
\theoremstyle{plain}

\theoremstyle{plain}
\newtheorem*{notation*}{Notation}
\theoremstyle{plain}
\newtheorem{notation}[defn]{Notation}

{\left\lbrace\begin{array}{@{}l@{}}}%
{\end{array}\right.}

\newcommand{\vertiii}[1]{{\left\vert\kern-0.25ex\left\vert\kern-0.25ex\left\vert #1 
    \right\vert\kern-0.25ex\right\vert\kern-0.25ex\right\vert}}

\newcounter{nootje}
\author{Salma Kuhlmann}

\author{Michele Serra}

\title{Automorphisms of valued Hahn groups}

\usepackage[autostyle]{csquotes}

\usepackage[backend=biber, style=alphabetic]{biblatex}
\makeindex
\begin{document}
\maketitle

\begin{abstract}
Let $ \Gamma $ be a totally ordered set and for all $ \gamma\in\Gamma $ let $ A_\gamma $ be an abelian group. The support of an element $ a=(a_\gamma)_{\gamma\in\Gamma}\in \prod_{\Gamma}A_\gamma $ is the set $ \mathrm{supp} (a) = \{\gamma:a_\gamma\neq 0\}. $
The Hahn sum $ \coprod_\Gamma A_\gamma $ (resp. Hahn product $ \mathbf{H}_\Gamma A_\gamma $) is the subgroup consisting of elements with finite (resp. well ordered support).
By a Hahn group we mean a group $ G $ such that $ \coprod_\Gamma A_\gamma \leq G \leq \mathbf{H}_\Gamma A_\gamma$. If for all $ \gamma\in\Gamma $,  $ A_\gamma $ is an ordered abelian group we endow $ \mathbf{H}_\Gamma A_\gamma $ with the lexicographic order.
Hahn groups, endowed with the canonical valuation $ v = v_{\min} $, play a fundamental role
in the classiﬁcation of valued groups. In this paper we study the group $ v\text{-}\mathrm{Aut} G $ (resp. $ o\text{-}\mathrm{Aut} G$) of valuation (resp. order)
preserving automorphisms of a Hahn group $ G$. Under
the assumption that $ G $ satisﬁes the \textit{canonical lifting property} we
prove a structure theorem decomposing $ v\text{-}\mathrm{Aut} G $ into a semidirect product of two notable subgroups: the subgroups of  \textit{internal} and \textit{external} automorphisms. We characterise
a class of groups satisfying the aforementioned lifting
property. For some special cases we provide matrix descriptions.
\end{abstract}


\section{Introduction}\label{sec:introduction}

In \cite{conrad}, Conrad describes the group of order preserving automorphisms of a Hahn sum as a certain group of matrices. Droste and G\"obel \cite{droste-goebel} extend this approach to balanced Hahn sum (see Definition \ref{def-hahn-sum-prod}). Hofberger \cite{hofbergerthesis} describes the group of automorphisms of Hahn fields as a semidirect product of so-called internal and external automorphisms. Building up on Hofberger's work we completed the description of the group of valuation preserving automorphisms of Hahn fields \cite{ours-fields}.

In this paper we extend the approach of \cite{ours-fields} to the study of $ \vAut G $ and $ \oAut G $ for a Hahn group $ G $.
We introduce the notion of (canonical) lifting property for a general Hahn group and, based on \cite{KKS-rayner-structures}, we prove that this is satisfied by a relatively large class of Hahn groups. We then decompose $ \vAut G $ (resp. $ \oAut G $) for Hahn groups satisfying the canonical lifting property. This extends the main result obtained in \cite{droste-goebel} (for the special case of balanced Hahn sums), to this large class of Hahn groups.

The paper is organised as follows.
In Section~\ref{sec:definitions-and-notation} we provide the definitions and notations that will be needed. In Section~\ref{sec:the-lifting-property} we introduce the important notion of \email{lifting property}. Subsection~\ref{sec:lifting-property-with-respect-to-the-rank} focusses on order preserving automorphisms and on the lifting property with respect to the rank (Definition~\ref{def:lifting-property-rank}). Theorem~\ref{lifting-principal-convex-sbgps} characterises automorphisms of the rank that lift by means of principal convex subgroups. We use this result to show that the lifting property with respect to the rank is not preserved by taking lexicographic sums (Example~\ref{es:lifting-not-preserved-lex-sums}).
In Subsection~\ref{section-lifting-skeleton} we study valuation preserving automorphisms and we introduce the lifting property with respect to the skeleton (Definition~\ref{def:lifting-property-skeleton}). The main result of the section is Theorem~\ref{hofberger-groups}, which provides a semidirect product decomposition of the $ \vAut G $ of a Hahn group $ G $ satisfying the lifting property w.r.t. the skeleton. In Subsection~\ref{sec:rayner-groups}, based on \cite{KKS-rayner-structures}, we introduce Rayner groups (Definition~\ref{def:rayner-group}) and characterise those that satisfy the lifting property w.r.t. the skeleton (Theorem~\ref{prop:rayner-groups-canonical-stable}).
Section~\ref{sec:matrices} focusses on Hahn sums, providing a description of the group of order preserving automorphisms as a group of matrices, generalising results of \cite{conrad} and \cite{droste-goebel}. Theorem~\ref{prop:oAutGamma-G-direct product} gives a decomposition of the subgroup of order preserving automorphisms that induce the identity on the rank while Theorem~\ref{thm:DG-semidirect-enhanced} provides the matrix group representation.

In \cite{kaplansky-lifting} we further investigate the subgroups of internal and external automorphisms of $ \oAut G $ of an ordered Hahn group $ G $, providing a finer decomposition of those groups. This analysis is of particular importance, since $ \oAut G $ appears as a factor of the automorphism group $ \vAut K $ of a valued Hahn field $ K $ with value group $G$ \cite{ours-fields}. Furthermore we focus on the subgroup of strongly additive (i.e., commuting with formal infinite sums) automorphisms. This latter subgroup is of particular importance for the study of derivations on valued Hahn fields (cf. \cite{derivations}).

\section{Definitions and notation}\label{sec:definitions-and-notation}
Let $ \Gamma $ be a (totally) ordered set and for all $ \gamma\in\Gamma $ let $ A_\gamma $ be an abelian group. For an element $ a = (a_\gamma)_{\gamma\in\Gamma} \in \prod_{\gamma\in\Gamma}A_\gamma $ we define the \emph{support} of $ a $
to be the set $ \supp(a) = \{\gamma\in\Gamma : a_\gamma\neq 0\} $.
\begin{defn}[Hahn group]
	\label{def-hahn-sum-prod}\label{der:skeleton}\label{def:rank}
	\begin{enumerate}[(i)]
	\item
	The \emph{Hahn sum} of the $ A_\gamma $'s, denoted by $ \coprod_{\gamma\in\Gamma}A_\gamma $ is the set of elements of the full cartesian product with finite support:
	\[
	\coprod_{\gamma\in\Gamma}A_\gamma = \left\lbrace a\in\prod_{\gamma\in\Gamma}A_\gamma : |\supp(a)|<\infty  \right\rbrace.
	\]
	\item 
	The \emph{Hahn product} of the $ A_\gamma $'s, denoted by $ \HH_{\gamma\in\Gamma}A_\gamma $ is the set of elements of the full direct product with well ordered support:
	\[
	\HH_{\gamma\in\Gamma}A_\gamma = \left\lbrace a\in\prod_{\gamma\in\Gamma}A_\gamma : \supp(a) \text{ is well ordered}  \right\rbrace.
	\]
	The Hahn sum is a subgroup of the Hahn product.
	\item 
	A group $ G $ comprised between them will be called a \emph{Hahn group} and the pair $ S(G):= [\Gamma;\{A_\gamma:\gamma\in\Gamma\}] $ will be called the \emph{skeleton} of $ G $.
	\item
	In the sequel we will also refer to $ \Gamma $ as the \emph{rank of} $ G $ (sometimes denoted by $ \rk G $) and to the groups $ A_\gamma $ as the \emph{components of} $ G $.
	\item 
	If for all $ \gamma\in\Gamma $ we have $ A_\gamma\simeq A $ for a given abelian group $A$ we say that $ G $ is a \emph{balanced Hahn group}.
	\item 
	If $ G $ is balanced, we will write $ S(G) = [\Gamma; A] $ and call it a \textit{balanced skeleton}.
	\end{enumerate}
\end{defn}
\begin{notation}
	We will use the following notation for the elements of $ G $:
	\[
	\sum_{\gamma\in\Gamma}a_\gamma\1_\gamma := (a_\gamma)_{\gamma\in\Gamma}.
	\]
\end{notation}

\begin{defn}
	Let $ G_1,G_2 $ be Hahn groups and let $ S(G_i) = [\Gamma_i:\{ A^i_\gamma : \gamma\in\Gamma_i \}],\ i=1,2 $ be the respective skeletons. An \emph{isomorphism} $ \tau:S(G_1)\ra S(G_2) $ consists of an isomorphism $ \tau_\Gamma:\Gamma_1\ra\Gamma_2 $ of ordered sets, and for all $ \gamma\in\Gamma_1 $ an isomorphism
	$ \tau_\gamma\colon A^1_\gamma\ra A^2_{\tau_\Gamma(\gamma)} $
	of ordered groups. We will also write
	$ \tau = [\tau_\Gamma;\{\tau_\gamma:\gamma\in\Gamma_1\}]. $\\
\end{defn}
An automorphism of $ S(G) $, for a Hahn group $G$, is hence defined in the obvious way and we will denote by $ \Aut S(G) $ the group of automorphisms of $ S(G) $ under composition.

Notice that there can be automorphisms of $\Gamma$ that cannot occur in any automorphism of the skeleton:
\begin{es}
	Let $ G = \coprod_\Z A_n $ where $ A_{2n} = \Q $ and $ A_{2n+1} = \R $ for all $ n\in\Z $. Then $\sigma_\Z\colon n\mapsto n+1 $ is an automorphism of the chain $ \Z $ that cannot appear in any automorphism of the skeleton, as there is no isomorphism between $ A_n $ and $ A_{n+1} $, for any $ n $.
\end{es}

\begin{defn}
	We call an automorphism $\xi$ of $\Gamma$ \emph{admissible} if, for every $\g\in\G$ there exists an isomorphism $\xi_\g\colon A_\g\to A_{\xi(\g)}$. In other words, $\xi$ is admissible if it can occur in an automorphism of the skeleton.
\end{defn}

\begin{rmk}
	\begin{enumerate}[(i)]
		\item One can easily verify, that the set $\adAut(\G,<)$ of admissible automorphisms forms a subgroup of $\Aut(\G,<)$ under composition.
		
		Let $\alpha,\beta\in\adAut(\G,<)$.
		Then, for all $\g\in\G $ the following two things happen:
		\begin{itemize}
			\item there exist isomorphisms $\alpha_\g\colon A_\g \to A_{\alpha(\g)}$ and $\beta_{\alpha(g)}\colon A_{\alpha(\g)} \to A_{\beta(\alpha(\g))}$ and therefore, there exists an isomorphism
		$$\beta_{\alpha(g)}\circ\alpha_\g \colon A_\g 
		\to A_{\beta(\alpha(\g))}.$$
		\item Because $\alpha$ is admissible, there exists an isomorphism $\alpha_{\alpha\inv(\g)} \colon A_{\alpha\inv(\g)} \to A_\g$. 
		And thus $(\alpha_{\alpha\inv(\g)})\inv \colon A_\g \to A_{\alpha\inv(\g)} $ is also an isomorphism.
		\end{itemize}
	\item If $G$ is balanced with $A_\g\simeq A$ for all $\g\in\G$, then every automorphism $\xi$ of $\Gamma$ is admissible, as we can take $\xi_\g=\id_A$ for all $\g\in\G$.
	We thus have an isomorphism
	\[
	\Aut S(G) \simeq \Aut \G \times \prod_{\g\in\G}\Aut A_\g.
	\]
	\end{enumerate}
\end{rmk}

For the remaining of this section, let us fix a Hahn group $ G $ with skeleton $ S(G) = [\Gamma;\{A_\gamma:\gamma\in\Gamma\}] $.

\begin{defn}\label{def:valuation-and-ordering}
	Since all elements of $G$ have well ordered support, we can define a valuation $ v_{\min} $ on $ G $ with value set $ \Gamma $ by
	\[
	v_{\min}\colon G \longrightarrow \Gamma\cup\{\infty\},\quad  v_{\min}(g) = \begin{cases}
	\infty &\text{if } g=0	\\
	\min\supp(g) &\text{if } g\neq 0
	\end{cases}
	\]
	We will call $ v_{\min} $ the \emph{canonical valuation on} $ G $ and if the context is clear we will simply denote it by $ v $.
\end{defn}
Unless otherwise stated we will always consider $ G $ endowed with its canonical valuation $ v $.

\begin{defn}\label{def:val-pres}
	An automorphism $ \sigma $ of $ G $ is said to be \emph{valuation preserving} if there exists a chain automorphism $ \sigma_{\Gamma} $ of $ \Gamma $ such that, for all $ g\in G $, we have $ \sigma_{\Gamma}(v(g)) = v(\sigma(g)) $.\\
	The valuation preserving automorphisms form a group under composition that we will denote by $ \vAut G $.
\end{defn}
\begin{rmk}[{\cite[Lemma~0.1]{salma-monograph}}]
	\label{rmk:induced-autom}
	An automorphism $ \sigma\in\Aut G $ is valuation preserving if and only if the map $ \sigma_\Gamma\colon\Gamma\ra\Gamma,\ v(x)\mapsto v(\sigma(x)) $
	is a well defined automorphism of chains. We will say that $ \sigma_{\Gamma} $ is the chain automorphism \emph{induced on $ \Gamma $ by $ \sigma $}.\qed
\end{rmk}
\begin{defn}\label{def:val-pres-iso}
	More generally, let $ G_1,G_2 $ be Hahn groups with skeletons $ S(G_i) = [\Gamma_i:\{ A^i_\gamma : \gamma\in\Gamma_i \}],\ i=1,2 $ and let $ v_1,v_2 $ be the canonical valuations on $ G_1 $ and $ G_2 $, respectively. An isomorphism $ \rho\colon G_1\to G_2 $ is said to be \emph{valuation preserving} if there exists an isomorphism of chains $ \rho_\Gamma\colon\Gamma_1\to\Gamma_2 $ such that, for all $ g\in G_1 $ we have $ \rho_\Gamma(v_1(g)) = v_2(\rho(g)) $.
\end{defn}

\begin{lemma}[{\cite[Lemma~0.2]{salma-monograph}}]
	\label{lemma:induced-autom-skeleton}
	Let us keep the notation of Definition~\ref{def:val-pres-iso}.
	A valuation preserving isomorphism $ \sigma\colon G_1\ra G_2 $ of Hahn groups naturally induces an isomorphism of the skeletons $ [\sigma_\Gamma;\{\sigma_\gamma:\gamma\in\Gamma_1\}] \colon S(G_1)\ra S(G_2) $ as follows:
	\begin{itemize}
		\item for any $ \g\in\Gamma_1 $ choose $ 0\neq a \in A_\g $ and set $ \sigma_\Gamma(\g) := v_2(\sigma(a\1_\g)) $ (the fact that $ \sigma $ is valuation preserving implies that $ v_2(\sigma(a\1_\g)) $ does not depend on the choice of $ a $);
	\item 
	for any $ \gamma\in\Gamma_1 $ and each $ g_\gamma\in A^1_\gamma $ set $ \sigma_\gamma(g_\gamma) = \sigma(g_\gamma\1_\g)_{\sigma_\Gamma(\gamma)} $.
	\end{itemize}	
	In particular, a valuation preserving automorphism $ \sigma\in\vAut G $ induces an automorphism $ \sigma_S = [\sigma_\Gamma;\{\sigma_\gamma:\gamma\in\Gamma\}]\in\Aut S(G) $ of the skeleton where $ \sigma_\Gamma\in\Aut(\G, <) $ is defined as in Remark~\ref{rmk:induced-autom} and $ \sigma_\gamma\colon A_\gamma\to A_{\sigma_\Gamma(\gamma)} $ is given by $ \sigma_\gamma(a) = (\sigma(a\1_\gamma))_{\sigma_\Gamma(\gamma)} $.
	\qed
\end{lemma}

\begin{defn}\label{def:lex-order}\label{def:oAut}
	If all the components $ A_{\gamma} $ are ordered abelian groups and we denote by $ <_\gamma $ the order relation on $ A_\gamma $, then we can order $ G $ lexicographically by setting, for all $ g \in G $,
	\[g>_{\lex}0\iff g_{v(g)}>_{v(g)}0. \]	
	An automorphism $ \sigma $ is said to be \emph{order preserving} if for all $ g\in G $ we have $ g>0\Ra \sigma(g)>0 $. The order preserving automorphisms also form a group under composition which will be denoted by $ \oAut G $.
\end{defn}

\begin{rmk}\label{rmk:oAut<vAut}
	Assume, moreover, that all the components $ A_\gamma $ be ordered \textit{archimedean} groups.
	Then an order preserving automorphism of $ G $ is automatically valuation preserving.
	So we have $ \oAut G\leq \vAut G\leq \Aut G. $\qed
\end{rmk}

\section{The lifting property}\label{sec:the-lifting-property}

Let $ G $ be a Hahn group with skeleton $ S(G) = [\Gamma;\{A_\gamma:\gamma\in\Gamma\}] $.
In Section~\ref{sec:definitions-and-notation} we saw that a valuation preserving automorphism $ \sigma\in\vAut G $ induces an automorphism $\sigma_S\in\Aut S(G) $ of the skeleton (Lemma~\ref{lemma:induced-autom-skeleton}) and, in particular, an automorphism $ \sigma_\Gamma\in\Aut(\G, <) $ of the rank (Remark~\ref{rmk:induced-autom}). In this section we will study under what conditions, given an automorphism of the skeleton (resp. rank) we can find an automorphism of the group that induces it. In Section~\ref{sec:lifting-property-with-respect-to-the-rank} we will study automorphisms of the chain $ \Gamma $ that lift to \emph{order preserving} automorphisms of $ G $ and in Section~\ref{section-lifting-skeleton} automorphisms of $ S(G) $ that lift to \emph{valuation preserving} automorphisms of $ G $.

\subsection{Lifting from the rank to $ \oAut G $}\label{sec:lifting-property-with-respect-to-the-rank}
Let $G$ be a Hahn group with skeleton $ S(G) = [\Gamma;\{A_\gamma:\gamma\in\Gamma\}] $ and let $ \GG = \HH_\Gamma A_\gamma $ be the corresponding Hahn product.
In this section we will assume that all the components $ A_\gamma $ of the skeleton of $ G $ be archimedean ordered abelian groups and we will consider the induced lexicographic order on $ G $ (Definition~\ref{def:lex-order}).
Let $\sigma \in \oAut G$. Then, by Remark~\ref{rmk:oAut<vAut}, $ \sigma\in\vAut G $ and therefore, it induces an automorphism $\sigma_{\Gamma}\in\Aut(\G, <) $ given by $\sigma_{\Gamma}(v(g)) = v(\sigma(g))$ (Remark~\ref{rmk:induced-autom}).
This gives rise to a group homomorphism
\begin{equation}\label{eq:phi-rank}
\Phi_\Gamma:\oAut G \lra \Aut(\G, <),\ \sigma\longmapsto \sigma_\Gamma.
\end{equation}
It is straightforward to verify that this is a group homomorphism.

\begin{defn}
	\label{def:rank-int}
	The kernel $ \ker\Phi_\Gamma $ of the map \eqref{eq:phi-rank}
	is a normal subgroup of $ \oAut G $ that we call the
	\emph{subgroup of internal o-automorphisms} of $G$.
	We will denote it  by $\Int\oAut G$.
\end{defn}
\begin{defn}\label{def:lifting-property-rank}
If the map $ \Phi_\Gamma $ defined in \eqref{eq:phi-rank} admits a section, i.e., an injective group homomorphism $ \Psi_\Gamma\colon \Aut(\G, <)\hookrightarrow \oAut G $ such that $ \Phi_\Gamma\circ\Psi_\Gamma = \id_{\Aut(\G, <)} $, we say that $ G $ has 	the \emph{lifting property with respect to $ \Psi_\Gamma $}.\\
For an automorphism $\tau \in \Aut(\G, <)$ we call $ \Psi_\Gamma(\tau) $ the \emph{lift of $ \tau $ with respect to $ \Psi_\Gamma $}.

\end{defn}
\begin{defn}
	\label{def:rank-ext}
	Assume $ G $ has the lifting property with respect to a fixed section $ \Psi_\Gamma $ of $ \Phi_\Gamma $.
	The subgroup $ \Psi_\Gamma(\oAut G) $ of $ \oAut G $ will be called the subgroup of $ \Psi_\Gamma $\emph{-external automorphisms of $G$} and denoted by ${\Psi_\Gamma}\text{-}\Ext\oAut G$.
\end{defn}
\begin{defn}\label{def:can-lift-prop-rank}
Assume $ G $ is balanced, say $ A_\gamma = A $ for all $ \gamma $. Then for the Hahn product $ \GG = \HH_\Gamma A $ there exists always the \emph{canonical section}
\[
\Psi^c_\Gamma\colon \Aut(\G, <)\to \oAut \GG,\ \tau\longmapsto \tilde\tau
\]
where $ \tilde\tau $ is defined by
\[
\tilde\tau\left( \sum_\gamma g_\gamma\1_\gamma\right) = \sum_\gamma g_\gamma\1_{\tau(\gamma)}.
\]
We call $ \tilde\tau $ the \emph{canonical lift of $ \tau $ to $ \GG $.}

\noindent If, moreover, $ \tilde\tau(G) = G $ we say that $ \tau $ \emph{lifts canonically} to an automorphism of $ G $ and call $ \tilde\tau|_G $ the \emph{canonical lift} of $ \tau $ to $ G $. If all automorphisms of $ \Gamma $ lift canonically we say that $ G $ has the \emph{canonical lifting property with respect to} $ \Gamma $.
\end{defn}

Now we give a characterisation of the automorphisms of the rank $ \Gamma $ that lift to an order-preserving automorphism or $ G $ by means of its principal convex subgroups (Definition~\ref{def:princ-convex-sbg}). We do this in
Theorem~\ref{lifting-principal-convex-sbgps}, after proving a preliminary lemma.
\begin{defn}\label{def:princ-convex-sbg}
	For an element $ \gamma\in\Gamma $, the \emph{principal convex subgroup of $ G $ associated to $ \gamma $} is the subgroup $ C_\gamma := \{x\in G : v(x)\geq \gamma \} $.
\end{defn}
\begin{lemma}
	\label{h-commutes-C}
	Let $\sigma\in\vAut G$ and let $g\in G$ and $C_g := \{x\in G:\ v(x)\geq v(g)\}$. Then we have:
	$$\sigma(C_g) = C_{\sigma(g)} = \{x:\ v(x) \geq v(\sigma(g))\}.$$
	\begin{proof}
		Let $ \sigma_\Gamma$ be the automorphism of the chain $\Gamma$ induced by $\sigma$.
		Let $x\in C_g$. Then
		$$v(x)\geq v(g) \Ra v(\sigma(x)) = \sigma_\Gamma (v(x)) \geq \sigma_\Gamma (v(g)) = v(\sigma(g))$$
		so $\sigma(x) \in C_{\sigma(g)}$, that is $\sigma(C_g) \subseteq C_{\sigma(g)}$.
		
		Vice versa, let $x\in C_{\sigma(g)}$. Then
		$$
		v(x) \geq v(\sigma(g)) \Ra
		v(\sigma^{-1}(x)) = \tilde \sigma^{-1}(v(x)) \geq \sigma_\Gamma^{-1}(v(\sigma(g)))= v(\sigma^{-1}(\sigma(g)))=v(g)
		$$
		because $\sigma_\Gamma^{-1}$ is also an automorphism of $\Gamma$.
		So $\sigma^{-1}(x) \in C_g$, that is $x\in \sigma(C_g)$.
		Hence $\sigma(C_g) \supseteq C_{\sigma(g)}$, which completes the proof.
	\end{proof}
\end{lemma}

\begin{theorem}
	\label{lifting-principal-convex-sbgps}
	Let $\tau\in\Aut(\G, <)$. Then $\tau$ lifts to an automorphism
	$\sigma \in \oAut G$ if and only if, for all $\gamma\in\Gamma$ there is an isomorphism
	$$
	\sigma_\gamma\colon C_\gamma \overset{\sim}{\lra} C_{\tau(\gamma)}
	$$
	such that 
	\begin{equation}
	\label{cond-gamma-delta}
	\gamma>\delta \Rightarrow \sigma_\delta|_{C_\gamma}=\sigma_\gamma.
	\end{equation}
\end{theorem}
\begin{proof}
	Assume that $\tau$ lifts to $ \sigma\in\oAut G $. Let $\gamma\in\Gamma$ and choose $g\in G$ such that
	$v(g) = \gamma$. Recall that, by definition, $v(\sigma(g)) = \tau (v(g))$.
	Then Lemma~\ref{h-commutes-C} gives
	$$
	C_\gamma =C_g = \{x:\ v(x)\geq \gamma = v(g)\}\simeq C_{\sigma(g)} = \{x:\ v(x) \geq \tilde \sigma(v(g)) = \tau(\gamma)\} = C_{\tau(\gamma)}.
	$$
	where the isomorphism is simply given by $\sigma$, which is order preserving, so
	condition \eqref{cond-gamma-delta} is satisfied.
	
	Vice versa, suppose that for all $\gamma\in\Gamma$ there is an isomorphism
	$$
	\sigma_\gamma\colon C_\gamma \overset{\sim}{\lra} C_{\tau(\gamma)}
	$$
	such that $\gamma>\delta \Rightarrow \sigma_\gamma(C_\gamma)\subset \sigma_\delta(C_\delta)$.
	We want to construct a lift $\sigma \in \oAut G$ of $\tau$. We define it as follows:
	for $x\in G$ with $v(x) = \gamma$, set $\sigma(x):=\sigma_\gamma(x)$. First let us show that
	$\sigma$ is a well defined group morphism. Let $x,y \in G$. We have two cases:\\
	\textbf{Case 1.}
	$v(x)\neq v(y)$, and we can assume $v(x)<v(y)$.
	Then $v(x+y)=v(x) =:\gamma$. So we have 
	$$
	\sigma(x-y) = \sigma_\gamma (x-y) = \sigma_\gamma(x) - \sigma_\gamma(y) = \sigma(x) - \sigma_{v(y)}(y) = \sigma(x) - \sigma(y)$$
	where in the second-last equality we used condition~\eqref{cond-gamma-delta}.\\
	\textbf{Case 2.} $v(x) = v(y) = \delta$. Then $v(x+y) =: \gamma \geq \delta$. Then
	$$
	\sigma(x-y) = \sigma_\gamma(x-y) = \sigma_\delta(x-y) = \sigma_\delta(x) - \sigma_\delta(y) = \sigma(x) - \sigma(y).
	$$
	Again, for the second equality we used condition~\eqref{cond-gamma-delta}.
	
	To show that $\sigma$ is an automorphism, we define its inverse. Similarly to what we did
	to define $\sigma$ we define $\sigma^{-1}$ as follows: for $x\in G$ with 
	$v(x) = \tau(\gamma)$ we set $\sigma^{-1}(x) = \sigma_\gamma^{-1}(x) $.
	Checking that this is a
	well defined group morphism is done the same way as it was done for $\sigma$. To see that this is indeed the inverse of $ \sigma $, let $x\in G$ with 
	$v(x)=\gamma$.
	Recall that $\sigma_\gamma(x) \in C_{\tau(\gamma)}$ and $\sigma_\gamma^{-1}: C_{\tau(\gamma)}\ra C_\gamma$. So
	$$
	\sigma^{-1}(h(x)) = \sigma^{-1}(\sigma_\gamma(x)) = \sigma_\gamma^{-1}(\sigma_\gamma(x)) = x
	$$
	and similarly $\sigma(\sigma^{-1}(x)) = x$. So $\sigma$ is an automorphism of $G$, which completes the proof.
\end{proof}

The following example shows that, in general, the lifting property is not preserved after taking lexicographic sums. 
\begin{es}\label{es:lifting-not-preserved-lex-sums}
	Let $\Gamma = \Q$ and $G = (\HH_{q<0}\Q) \amalg (\coprod_{q\geq0}\Q)$. So $\Gamma = \rk G = \Q$. Consider the automorphism of the chain $\Q$ defined by $\tilde h (q)= q+2$ and assume that $\tilde h$ lifts to an automorphism $h$ of $G$. Let $g\in G$ be such that $v(g) = -1/2$. So $v(h(g)) = \tilde h(-\frac{1}{2})=\frac{3}{2}>1$. By Lemma~\ref{h-commutes-C} we have $C_g \simeq C_{h(g)}$.  Now $C_{h(g)} = \{x:\ v(x) \geq \frac{3}{2}\}\subseteq\coprod_{q>0}\Q$ is countable. But 
	$C_g = \{x:\ v(x)\geq -\frac{1}{2}\}$ is uncountable. To prove this, we show that
	the uncountable set $\{0,1\}^\N$ of functions $f:\N\ra\{0,1\}$ can be embedded
	into $C_g$. We construct the embedding as follows: to each function $f:\N\ra\{0,1\}$
	we associate the element $f^* = (f^*_q)_{q<0}\in\HH_{q<0}\Q$ defined by
	$$
	f^*_q =
	\begin{cases}
	f(n) \text{ if } q = -\frac{1}{2+n}\\
	0 \text{ otherwise.}
	\end{cases}
	$$
	Notice that, for all $f\in \{0,1\}^\N$, the element $f^*$ belongs to $C_g$: indeed
	we have $\supp (f^*) \subseteq \{-\frac{1}{2+n}|n\in\N\}$ which is well ordered,
	so $f^*\in\HH_{q<0}\Q$ and $v(f^*)\geq -\frac{1}{2}$. Moreover,
	the map $f\mapsto f^*$ is injective, for if $f\neq g$ then there exists $n\in\N$ such
	that $$f^*_{-\frac{1}{2+n}} = f(n)\neq g(n) = g^*_{-\frac{1}{2+n}}.$$\qed
\end{es}

This last example shows that, even if two Hahn groups have the lifting property, after we take their lexicographic sum not all automorphism of the rank will lift. The next proposition characterises those that do.

\begin{prop}\label{prop:condition-lifting-preserved-lex}
	Let $G_1,G_2$ be two groups with the lifting property (with respect to the rank) and consider $G =G_1\amalg G_2$. Then $\rk G=\rk G_1 + \rk G_2$. Let $\tau$ be an automorphism of $\rk G$ such that $\tau|_{\rk G_i}\in\Aut(\rk G_i,<) $ for $i=1,2$.
	Then $\tau$ lifts to an automorphism of $G$.
\end{prop}
\begin{proof}
	Let $ \Gamma:= \rk G $ and $ \Gamma_i :=\rk G_i $ for $ i=1,2 $. Let $ \tau \in\Aut(\G, <) $ be such that $ \tau_i :=\tau|_{\Gamma_i}\in \Aut(\G, <)_i $ for $ i=1,2 $. Since both $ G_i $'s have the lifting property, let $ \sigma_i\in\vAut G_i $ be a lift of $ \tau_i $ and define $ \sigma\colon G\ra G $ through $ \sigma(g_1\1_1 + g_2\1_2):=\sigma_1(g_1)\1_1+\sigma_2(g_2)\1_2 $. It is straightforward to verify that this is an automorphism of $ G $.
\end{proof}

\subsection{Lifting from the skeleton to $ \vAut G $}
\label{section-lifting-skeleton}
In this section we consider the \emph{lifting property with respect to the skeleton}, and ask under what conditions an automorphism of the skeleton of a Hahn group lifts to an automorphism of the group.
Recall that (Lemma~\ref{lemma:induced-autom-skeleton}) a valuation preserving automorphism $ \sigma\in\vAut G $ induces an automorphism $ \sigma_S \in\Aut S(G) $ of the skeleton.
This gives rise to a map, that we will denote by $ \Phi_S $, in analogy to the one in \eqref{eq:phi-rank}:
\begin{equation}\label{eq:phi-skeleton}
	\Phi_S\colon \vAut G \lra \Aut S(G),\ \sigma\longmapsto[\sigma_\Gamma;\{\sigma_{\gamma}:\gamma\in\Gamma\}].
\end{equation}
It is straightforward to verify that $ \Phi_S $ is a group homomorphism.

\begin{defn}\label{def:internal}
	The kernel of the map $ \Phi_S $ defined in \eqref{eq:phi-skeleton} is a normal subgroup of $ \vAut G $ called \emph{the group of internal v-automorphisms} and denoted by $ \Int\vAut G $.
	
	If the context is clear we will drop the ``$ v $'' from the notation and terminology.
\end{defn}
\begin{rmk}\label{rmk:internal-by-valuation}
	An automorphism $\sigma\in\vAut G$ is internal if and only if, for all $g\in G$, we have
	\begin{equation}
	\label{eq:internal-by-valuation}
	v(g)=v(\sigma(g))\quad\text{ and }\quad g_{v(g)}=\sigma(g)_{v(g)}.
	\end{equation}
	Indeed, conditions \eqref{eq:internal-by-valuation} are equivalent to $ \Phi(\sigma)=\id_{S(G)} $.
	\qed
\end{rmk}

\begin{defn}\label{def:lifting-property-skeleton}
If the map $ \Phi_S $ defined in \eqref{eq:phi-skeleton} admits a section, i.e., an injective group homomorphism $ \Psi_S\colon \Aut S(G)\hookrightarrow \vAut G $ such that $ \Phi_S\circ\Psi_S = \id_{\Aut S(G)} $, we say that $ G $ has the \emph{lifting property with respect to $ \Psi_S $}.
For $ \tau\in\Aut(S(G)) $, the automorphism $ \Psi_S(\tau)\in\vAut G $ is called \emph{the lift of $ \tau $ with respect to $ \Psi_S $}.
\end{defn}
\begin{defn}
\label{def:skeleton-ext}
Assume $ G $ has the lifting property with respect to a fixed section $ \Psi_S $ of $ \Phi_S $.
The subgroup $ \Psi_S(\Aut S(G)) $ of $ \vAut G $ will be called the subgroup of $ \Psi_S $\emph{-external automorphisms of $G$} and denoted by ${\Psi_S}\text{-}\Ext\Aut G$.
If no confusion can arise we will just write $ \Ext\Aut G $.
\end{defn}

\begin{rmk}
	If $ G $ has the lifting property with respect to a given section $ \Psi_S $ then the homomorphism $ \Phi_S $
	defined in \eqref{eq:phi-skeleton} is surjective and every $ \tau\in\Aut S(G) $ lifts to an automorphism $ \sigma\in\vAut G $. The first isomorphism theorem thus yields
	\[
	\Aut S(G)\simeq\frac{\vAut G}{\Int\Aut G}.
	\]
\qed\end{rmk}

In \cite[Section~3]{ours-fields} we provide a decomposition of the group of valuation preserving automorphisms of a Hahn field with the lifting property into the semidirect product of its subgroups of internal and external automorphisms. Here we generalise this result to Hahn groups with the lifting property.

\begin{theorem}
	\label{hofberger-groups}
	Let $G$ be a Hahn group with the lifting property with respect to a given section $ \Psi_S $ of $ \Phi_S $.
	Then we can decompose $ \vAut G $ into the (inner) semidirect product of its subgroups of internal and external automorphisms:
	$$\vAut G = \Int\Aut G \rtimes \Ext\Aut G.$$
\end{theorem}
\begin{proof}
	Let us consider
	the sequence
	$$
	\xymatrix{ 
		\Int\Aut G \ar[r]^-\iota &
		\vAut G \ar[r]^-{\Phi_S} &
		\Aut S(G) \ar@{-->}@/^1pc/[l]^{\Psi_S}
	}
	$$
	where $\iota$ is the canonical embedding. By definition of $ \Int\Aut G $, we have $\im \iota = \ker\Phi$, so the sequence is exact. Therefore (see \cite[Theorem~3.3]{conrad-note-splitting}) we have
	\[
	\vAut G = \im \iota \rtimes \im \Psi = \ker\Phi\rtimes \im\Psi 
	= \Int\Aut G \rtimes \Ext\Aut G.
	\]
\end{proof}

Analogously to the previous section, we can define the canonical lifting property with respect to the skeleton: 

\begin{defn}\label{def:can-lift-prop-skeleton}
	For the Hahn product $ \GG = \HH_\Gamma A_\gamma $ there exists always the \emph{canonical section}
	\[
	\Psi^c_S\colon \Aut S(G)\to \vAut \GG,\ \tau = [\tau_\Gamma;\{\tau_\gamma : \gamma\in\Gamma \}]\longmapsto \tilde\tau
	\]
	where $ \tilde\tau $ is defined by
	\[
	\tilde\tau\left( \sum_\gamma g_\gamma\1_\gamma\right) = \sum_\gamma \tau_\gamma(g_\gamma)\1_{\tau_\Gamma(\gamma)}.
	\]
	We call $ \tilde\tau $ the \emph{canonical lift of $ \tau $ to $ \GG $.}
	
	\noindent If, moreover, $ \tilde\tau(G) = G $ we say that $ \tau $ \emph{lifts canonically} to an automorphism of $ G $ and call $ \tilde\tau|_G $ the \emph{canonical lift of $ \tau $ to $ G $}. If all automorphisms of $ S(G) $ lift canonically we say that $ G $ has the \emph{canonical lifting property with respect to} $ S(G) $.
\end{defn}

In the next section we introduce an interesting class of Hahn groups, that we will call \emph{Rayner groups} (Definition~\ref{def:rayner-group}) in which we can characterise those groups that have the canonical lifting property.

\subsection{Rayner groups}\label{sec:rayner-groups}
A family $ \cF $ of subsets of $ \Gamma $
is said to be a \emph{group family} 
(cf. \cite[Paragraph 2]{rayner1968})
if the following properties are satisfied:
\begin{enumerate}[(i)]
	\item the elements of $ \cF $ are well ordered subsets of $ \Gamma $;
	\item $ A,B\in\cF \Ra A\cup B\in\cF $;
	\item $ A\in\cF\wedge B\subseteq A \Ra B\in\cF $.
\end{enumerate}
\begin{prop}[{\cite[Theorem~1]{rayner1968}}]\label{prop:rayner-family-yield-groups}
	If $ \cF $ is a group family then the set $ \GG(\cF) $ of elements of $ \GG $ whose supports belong to $ \cF $ is a subgroup of $ \GG $.
\end{prop}
\begin{proof}
	By property (i) all elements of $ \cF $ are well ordered so they can occur as supports of elements of $ \GG $. Since for all $ g\in \GG $ we have $ \supp(g) = \supp(-g) $ then $ g\in\GG(\cF) $ implies $ -g\in\GG(\cF) $. If $ g,h\in\GG(\cF) $, then $ \supp(g+h)\subseteq \supp(g) \cup \supp(h) $: the latter is in $ \cF $ by (ii) and hence by (iii) the former is too. So $ g+h\in\GG(\cF) $. It follows then that also $ 0\in\GG(\cF) $ so $ \GG(\cF) $ is indeed a subgroup of $ \GG $.
\end{proof}
\begin{defn}\label{def:rayner-group}
	Groups obtained as in Proposition~\ref{prop:rayner-family-yield-groups} will be called \emph{Rayner groups}. These generalise an analogous notion (Rayner fields) introduced by Rayner in \cite{rayner1968} in the context of Hahn fields. Further \emph{Rayner structures} are introduced and studied in \cite{KKS-rayner-structures}.
\end{defn}

The next result gives a necessary and sufficient condition for a Rayner group to satisfy the canonical lifting property.
\begin{theorem}\label{prop:rayner-groups-canonical-stable}
	Let $ \GG $ be the Hahn product over the skeleton $ [\Gamma;\{ A_\gamma:\gamma\in\Gamma \}] $. Let $ \cF $ be a group family with respect to $ \Gamma $. Then the Rayner group $ G = \GG(\cF) $ has the canonical lifting property with respect to the skeleton if and only if $ \cF $ is stable under the action of $ \Aut S(G) $, that is, for all $ [\bar\tau;\{ \tau_\gamma:\gamma\in\Gamma \}]\in\Aut S(G) $ and for all $ T\in\cF $ we have $ \bar\tau(T)=T $.
\end{theorem}
\begin{proof}
	Assume that $ \cF $ be stable under $ \Aut S(G) $ and let $ g = \sum g_\gamma\1_\gamma\in G $. Let also $ \tau=[\bar\tau;\tau_\gamma] $ be an automorphism of $ S(G) $. Then $ \supp\left( \sum\tau_\gamma(g_\gamma)\1_{\bar\tau(\gamma)}\right) = \{ \bar\tau(\gamma):\gamma\in\supp(g) \} = \bar\tau(\supp(g)) $ which, by our assumption, belongs to $ \cF $. Thus $ \sum\tau_\gamma(g_\gamma)\1_{\bar\tau(\gamma)} \in G $ and so $ G $ has the canonical lifting property with respect to the skeleton.
	
	Vice versa, assume that $ G $ has the canonical lifting property and let $ T\in\cF $. Let $ g = \sum g_\gamma\1_\gamma \in G $ be such that $ \supp(g)=T $. Then, for all $ [\bar\tau;\tau_\gamma] $ we have $ h:= \sum\tau_\gamma(g_\gamma)\1_{\bar\tau(\gamma)}\in G $, so in particular
	$ \supp(h) \in \cF $. So $ \cF $ is stable under action of $ \Aut S(G) $.
\end{proof}

\begin{defn}
	Let $ \kappa $ be an infinite regular cardinal and let $ \cF_\kappa $ be the family of all well ordered subsets of $ \Gamma $ of cardinality smaller than $ \kappa $. The $ \kappa $-bounded subgroup of $ \GG $ consists of all elements of $ \GG $	with support of cardinality less than $ \kappa $, and will be denoted by $ \GG_\kappa $.
	
	\noindent Notice that the Hahn sum and the Hahn product are special cases of $ \kappa $-bounded subgroups, with $ \kappa = \aleph_0 $ and $ \kappa = |\GG|^+ $ respectively.
\end{defn}

\begin{cor}\label{cor:k-bdd-have-canonical-lifting}
	Let $ \kappa $ be an infinite regular cardinal. Then $ \GG_\kappa $ has the canonical lifting property. In particular, $ \GG $ and $ \coprod_\Gamma A_\gamma $ have the canonical lifting property.
\end{cor}
\begin{proof}
	By Proposition~\ref{prop:rayner-groups-canonical-stable} it suffices to show that the family $ \cF_\kappa $ of well ordered subsets of $ \Gamma $ is stable under $ \Aut S(G) $. Let $ [\bar\tau;\{ \tau_\gamma:\gamma\in\Gamma \}]\in\Aut S(G) $ and let $ T $ be a well ordered subset of $ \Gamma $ of cardinality less than $ \kappa $. Then $ \bar\tau(T) $ will have the same cardinality as $ T $, because $ \bar\tau $ is an automorphism of $ \Gamma $. Since it is also order preserving then $ \bar\tau(T) $ will also be well ordered, hence $ \bar\tau(T)\in\cF_\kappa $.
\end{proof}

Analogous results to the ones contained in this section hold for Rayner fields: see \cite[Section~3.2]{ours-fields}.

\section{Automorphism groups as groups of matrices}\label{sec:matrices}
In this section we are going to generalise and sharpen results of Conrad \cite{conrad} and Droste-G\"obel \cite{droste-goebel}. The first describes the group of order preserving automorphisms of a Hahn sum as a certain group of triangular matrices. The latter obtain a semidirect product decomposition, which is different from the one we gave in Theorem~\ref{hofberger-groups} -- we will explain how the two relate.
 
Let $ \Gamma $ be a chain, $ \{A_\gamma:\gamma\in\Gamma\} $ a family of ordered abelian groups and $ G = \coprod_{\gamma\in\Gamma}A_\gamma $ the Hahn sum. Let $ \End A_\gamma $ be the ring of endomorphisms of $ A_\gamma $ (with pointwise addition and composition). For all $ \alpha,\beta\in\Gamma $ let $ H_{\alpha\beta} = \Hom(A_\beta,A_\alpha) $ be the group of homomorphisms from $ A_\beta $ into $ A_\alpha$ (with pointwise addition). Let $ \Delta $ be the set of all $ \Gamma{\times}\Gamma $-matrices $ (\sigma_{\alpha\beta}) $ where
\begin{enumerate}[(i)]
	\item $ \sigma_{\alpha\alpha}\in\End A_\alpha $;
	\item $ \sigma_{\alpha\beta}\in H_{\alpha\beta} $;
	\item for every $ \beta $ and for all $ a\in A_\beta $ we have $ \sigma_{\alpha\beta}(a)=0 $ for all but finitely many $ \alpha $.
\end{enumerate}
Then $ \Delta $ forms a ring with respect to the usual matrix addition and multiplication (condition (iii) ensures that the product be well defined).
\begin{rmk}\label{rmk:isom-End-Delta}
There is an isomorphism between the rings $ \End G $ and $\Delta $.
Indeed, for $ a=\sum a_\gamma\1_\gamma\in G $ and $ (\sigma_{\alpha\beta})\in\Delta $ we can consider the row vector $ (a_\gamma) $ and multiply it on the left to get
\begin{equation}\label{eq:formula-matrix-multiplication}
(a_\gamma)(\sigma_{\alpha\beta}) = \left(\sum_{\alpha\in\supp(a)} \sigma_{\alpha\beta}(a_\alpha)\right)_{\beta\in\Gamma} =: (b_\beta)_{\beta\in\Gamma}
\end{equation}
where the sum on the right hand side makes sense because the $ \supp (a) $ is finite.
Then it is clear that the map $ \sigma: \sum a_\gamma\1_\gamma\mapsto \sum b_\gamma\1_\gamma $ is an endomorphism of $ G $ induced by $ (\sigma_{\alpha\beta}) $.

Vice versa, if $ \sigma\in\End G $ and $ a=\sum a_\gamma\1_\gamma\in G $
let $ b=\sum b_\gamma\1_\gamma := \sigma(a) $. Then for all $
\alpha\in\Gamma $ define $ \sigma_{\alpha\beta}(a_\alpha):= b_\beta $.
Since $ b\in G $ then $ \supp (b) $ is finite and therefore $
\sigma_{\alpha\beta} $ satisfies condition (iii) above. Hence $ (\sigma_{\alpha\beta})\in \Delta $ and it is easy to verify that the correspondence $ \sigma \mapsto (\sigma_{\alpha\beta}) $ is an isomorphism of $ \End G $ onto $ \Delta $.
Thus the group $ \Aut G $ corresponds isomorphically to the group of units in $ \Delta $. Here by $ \Aut G $ we mean \emph{all} automorphisms of $ G $ as an abelian group, order (reps. valuation) preserving or otherwise. \qed
\end{rmk}

Now we see how we can characterise the order preserving automorphisms.
Let us consider the multiplicative monoid $ T $ of all matrices $ (\sigma_{\alpha\beta}) \in\Delta $ such that
\begin{enumerate}[(i')]
	\item $ (\sigma_{\alpha\beta}) $ is lower triangular;
	\item for all $ \alpha\in\Gamma $, $ \sigma_{\alpha\alpha}\in \Aut (A_\alpha,<) $.
\end{enumerate}
The order on $ G $ is the lexicographic one (cf. Definition~\ref{def:valuation-and-ordering}).
\begin{lemma}\label{lemma:conrad-lemma0}
A matrix $ (\sigma_{\alpha\beta})\in T $ induces, via the correspondence described in Remark~\ref{rmk:isom-End-Delta}, an order preserving endomorphism $ \sigma $ on $ G $.
In particular, an invertible matrix in $ T $ induces an order preserving automorphism of $ G $.
\end{lemma}
\begin{proof}
Let $ a = \sum_\gamma a_\gamma\1_\gamma >0 $ with $ v(a)=\gamma $. Then
\begin{align*}
\sigma(a) = (a)_\alpha(\sigma_{\alpha\beta}) &= \left(\sum_{\alpha\geq\gamma}\sigma_{\alpha\beta}(a_\alpha) \right)_{\beta\in\Gamma} \\&= \left(\sum_{\alpha\geq\gamma}\sigma_{\alpha\beta}(a_\alpha) \right)_{\beta\geq\gamma} = \sigma_{\gamma\gamma}(a_\gamma)+\left(\sum_{\alpha\geq\gamma}\sigma_{\alpha\beta}(a_\alpha) \right)_{\beta>\gamma}>0
\end{align*}
where the third equality holds because $ \sigma_{\alpha\beta} $ is lower triangular and the final inequality is due to the fact that, by (ii'), $ \sigma_{\gamma\gamma} $ is order preserving.
\end{proof}
\begin{cor}\label{cor:conrad-lemma0}
	Let $ U $ denote the group of units (the invertible matrices) in $ T $.
	Then $ U $ embeds into $ \oAut G $.\qed
\end{cor}
\begin{rmk}
From the proof of Lemma~\ref{lemma:conrad-lemma0} it is clear that all $ o $-automorphisms of $ G $ induced by the elements $ (\sigma_{\alpha\beta})\in U $ induce the identity on $ \Gamma $. If, moreover, such an automorphism also induces the identity on each component then it will be an internal automorphism. Again, from the formula in the proof of Lemma~\ref{lemma:conrad-lemma0}, we see that this happens if and only if $ \sigma_{\alpha\alpha}=1 $ for all $ \alpha\in\Gamma $.
\qed
\end{rmk}
\begin{theorem}\label{prop:oAutGamma-G-direct product}
	Let us denote by $ \oAut_\Gamma G $ the order preserving automorphisms of $ G $ that induce the identity on the chain $ \Gamma $. 
	Similarly, let $ \Aut_\Gamma S(G) $ denote the group of automorphisms of the skeleton whose component on $ \Gamma $ is the identity.
	Then we have
	\[
	\oAut_\Gamma G \simeq \Int\Aut G \rtimes \Aut_\Gamma S(G).
	\]
	It follows, in particular, that $ \oAut_\Gamma G $ forms a group.
\end{theorem}
\begin{proof}
	Let $ U^1 $ be the subgroup of $ U $ given by all the \emph{unitriaugular} matrices of $ T $ (all the diagonal entries are 1). Then $ \Int\Aut G \simeq U^1 $. When we multiply two diagonal matrices, the elements on the diagonal of the product are the products of the corresponding diagonal entries. Hence $ U^1 $ is a normal subgroup of $ U $. Moreover, every element $ (\sigma_{\alpha\beta})\in U $ is the product of an element in $ U^1 $ and a diagonal matrix: namely $ (\sigma_{\alpha\beta})=(\sigma_{\alpha\beta})^1(\sigma_{\alpha\beta})^d $ where $ (\sigma_{\alpha\beta})^1 $ is obtained by $ (\sigma_{\alpha\beta}) $ replacing all diagonal elements by 1 and $ (\sigma_{\alpha\beta})^d $ is the diagonal matrix with the diagonal entries of $ (\sigma_{\alpha\beta}) $ on its diagonal. And finally, if we denote by $ U^d $ the subgroup of diagonal matrices in $ U $ we have $ U^d\cap U^1=\{ 1 \} $. Hence $ U = U^1\rtimes U^d $.
\end{proof}
The last proposition gives us a matrix description of the group of automorphisms that induce the identity on $ \Gamma $. 
But we know that every automorphism of the skeleton (there might be some where the component on $ \Gamma $ is not trivial) lifts to an automorphism of $ G $, because Hahn sums have the (canonical) lifting property. So, in order to completely describe $ \oAut G $, we need to take into account the lifts of automorphisms of the skeleton with non trivial component on $ \Gamma $. Let $ \Aut_\Gamma'S(G)  $ be a complement of $ \Aut_\Gamma S(G) $ in $\Aut S(G) $. 
We have:
\begin{theorem}\label{thm:DG-semidirect-enhanced}
	We have
	\[
	\oAut G = \left( U^1 \rtimes U^d \right)
	\rtimes \Aut_\Gamma'S(G).
	\]
\end{theorem}
\begin{proof}
	By Theorem~\ref{prop:oAutGamma-G-direct product} we have
	$ \oAut_\Gamma G \simeq \Int\Aut G \times \Aut_\Gamma S(G) \simeq U^1 \rtimes U^d $ hence it suffices to show that $ \oAut G \simeq \oAut_\Gamma G \rtimes \Aut_\Gamma'S(G) $. We notice immediately that the two groups have trivial intersection and that $ \oAut_\Gamma G $ is normal inside $ \oAut G $. To show that they generate $ \oAut G $ we take an automorphism $ \sigma\in\oAut G $. This will induce an automorphism $ \tau\in\Aut S(G) $ which might or might not fix $ \Gamma $. In any case, if we lift $ \tau $ to a $ \tilde{\tau}\in\oAut G $ we certainly have that $ \sigma\tilde\tau^{-1} $ induces the identity on $ \Gamma $. So $ \sigma = (\sigma\tilde\tau^{-1})\tilde\tau $.
\end{proof}

Theorem~\ref{thm:DG-semidirect-enhanced} improves \cite[Corollary~3.2]{droste-goebel} in two ways: it provides a generalisation from the balanced case to that of a Hahn sum over an arbitrary skeleton and gives a more precise description of $ \oAut G $ by identifying the internal automorphisms within the group of matrices $ U $.

\addcontentsline{toc}{section}{\refname}
\printbibliography

\bigskip

\textsc{Salma Kuhlmann}, Fachbereich Mathematik und Statistik, Universität Konstanz, Germany\\
\textit{E-mail address}: \texttt{salma.kuhlmann@uni-konstanz.de}

\medskip

\textsc{Michele Serra}, Fakultät für Mathematik, Technische Universität Dortmund, Germany\\
\textit{E-mail address}: \texttt{michele.serra@tu-dortmund.de}
\end{document}